\documentclass[12pt, reqno]{article}
\title{The half plane UIPT is recurrent}
\author{Omer Angel \and Gourab Ray}
\date{January 2016}  %{\today}

\AtEndDocument{
  \bigskip
  \small
  \textsc{Omer Angel} \par
  \textsc{Department of Mathematics, University of British Columbia} \par
  \textit{Email:} \texttt{angel@math.ubc.ca} \par

  \medskip

  \textsc{Gourab Ray} \par
  \textsc{Statistical Laboratory, University of Cambridge}\par
  \textit{Email:} \texttt{rg508@cam.ac.uk} \par
}

\usepackage{amsmath,amssymb,amsfonts,amsthm}

\usepackage{color}
\usepackage{enumitem}

\usepackage{hyperref}
\hypersetup{
    colorlinks=true,
    linkcolor=blue,
    citecolor=red,
    urlcolor=blue,
    pdfborder={0 0 0}
}
\usepackage{graphicx}
\usepackage[font=sf, labelfont={sf,bf}, margin=1cm]{caption}
\usepackage{cleveref}

\crefname{theorem}{Theorem}{Theorems}
\crefname{mainthm}{Theorem}{Theorems}
\crefname{thm}{Theorem}{Theorems}
\crefname{lemma}{Lemma}{Lemmas}
\crefname{lem}{Lemma}{Lemmas}
\crefname{remark}{Remark}{Remarks}
\crefname{prop}{Proposition}{Propositions}
\crefname{defn}{Definition}{Definitions}
\crefname{corollary}{Corollary}{Corollaries}
\crefname{conjecture}{Conjecture}{Conjectures}
\crefname{question}{Question}{Questions}
\crefname{chapter}{Chapter}{Chapters}
\crefname{claim}{Claim}{Claims}
\crefname{section}{Section}{Sections}
\crefname{figure}{Figure}{Figures}

\theoremstyle{plain}
\newtheorem{mainthm}{Theorem}
\newtheorem{thm}{Theorem}[section]
\newtheorem{lemma}[thm]{Lemma}

\newtheorem{corollary}[thm]{Corollary}
\newtheorem{prop}[thm]{Proposition}

\theoremstyle{definition}

\theoremstyle{remark}

\numberwithin{equation}{section}

\renewcommand{\P}{\mathbb P}
\newcommand{\Z}{\mathbb Z}
\newcommand{\E}{\mathbb E}
\newcommand{\R}{\mathbb R}

\newcommand{\eps}{\varepsilon}

\usepackage{bbm}

\renewcommand{\H}{\mathbb H}

\newcommand{\Bb}{\mathbb B^{\text{bias}}}
\newcommand{\B}{\mathbb B}
\newcommand{\Sk}{\text{Skel}}
\newcommand{\cF}{\mathcal F}
\newcommand{\cC}{{\mathcal C}}
\newcommand{\Res}{R_{\mathrm{eff}}}

\begin{document}

\maketitle

\begin{abstract}
  We prove that the half plane version of the uniform infinite planar
  triangulation (UIPT) is recurrent.

  The key ingredients of the proof are a construction of a new full plane
  extension of the half plane UIPT, based on a natural decomposition of the
  half plane UIPT into independent layers, and an extension of previous
  methods for proving recurrence of weak local limits (still using circle
  packings).
\end{abstract}

\section{Introduction}

The \textbf{half plane uniform infinite planar triangulation}, abbreviated
as the \textbf{HUIPT} below, is a random planar triangulation, closely
related to the well-known and extensively studied \textbf{uniform infinite
  planar triangulations} (\textbf{UIPT}), but with the topology of the half
plane.  The HUIPT is an interesting object in its own right, and in some
ways is nicer than the UIPT.  For example, it possesses a simpler form of
the domain Markov property (detailed definition are provided in
\cref{sec:background}).  The problem of establishing the recurrence of the
UIPT had been open for many years.  This was a motivation for the seminal
work of Benjamini and Schramm \cite{BeSc}, and was resolved in recent work
of Gurel-Gurevich and Nachmias \cite{GN12}.  However, recurrence of the
HUIPT does not follow from their work, as it is not known if it is possible
to realize the HUIPT as a subgraph of the UIPT (indeed, there are possible
arguments that it is not possible to have such a coupling).  In this
article we establish the recurrence of the half-plane UIPT.

\begin{mainthm}\label{thm:main}
  The simple random walk on the half plane uniform infinite planar
  triangulation is almost-surely recurrent.
\end{mainthm}

While our proof  incorporates some ideas from \cite{BeSc,GN12}, 
new methods are also needed.  A crucial ingredient in those works is that the
graphs under consideration are weak local limits of finite planar graphs,
with a root that is chosen uniformly among all vertices.
After embedding the graphs in
carefully chosen manner in the plane, this leads to a fundamental Lemma on
geometry of arbitrary point sets in the plane \cite[Lemma~4.2]{BeSc}.  A
quantitative version of this lemma \cite[Lemma~3.4]{GN12} was exploited to
prove the recurrence of the UIPT, and is used in this work as well (see
\cref{lem:delta_s} below).

Crucially, the methods of \cite{BeSc,GN12} do not apply directly, since the
HUIPT is not a weak local limit of finite planar graphs.  The original
construction of the HUIPT is as a weak limit of uniform triangulations with
boundaries where the root is restricted to the boundary.  In particular,
the root is not a uniform vertex.  The main novelty of this work lies in
the technique used to overcome this obstacle.  Along the way we obtain a
certain random full plane map $M$ which we call the \textbf{layered UIPT}.
The layered UIPT contains the HUIPT as a subgraph.  We believe $M$ to be of
independent interest and to have further applications.  We prove that $M$
is recurrent which implies that the HUIPT is recurrent.

Another difficulty stems from the fact that (unlike the UIPT), the HUIPT is
not stationary for the simple random walk.  Indeed, viewed from the random
walker, the HUIPT should converge in distribution w.r.t.\ the local
topology to the UIPT, as the walker will typically be far from the boundary.
The map $M$ we introduce is not stationary
itself, but there is a certain local modification of $M$ which is
stationary, and even reversible.  Thus in a certain sense, the map $M$ can
be seen as a stationary reversible version of the HUIPT.  (A random rooted
graph 
$(G,\rho)$ is stationary reversible with respect to simple random walk
$\{\rho,X_1,\dots \}$ if the law of the doubly marked graph $(G,\rho,X_1)$
is the same as the law of
$(G,X_1,\rho)$ see \cite{BC2011}; reversibility and the related property of
unimodularity has been exploited in the
past to great advantage \cite{BC2011,BLPS99,ANR14,BC11,AHNR14,AHNR15,URT}.) 

Finally, a central tool we use is a decomposition of the HUIPT and of the
layered UIPT into
independent layers (see \cref{sec:layer}).  An analogous decomposition was
used by Krikun \cite{kriUIPT,krikun2005local} for the UIPQ.  However, the
domain Markov property of HUIPT gives this decomposition a particularly
elegant structure.  Such a decomposition has great potential for the study
of random maps.  A forthcoming recent work of Curien and Le Gall
\cite{CuLeGall15} analyzes first passage percolation and other
perturbations of the metric structure of the UIPT via such a decomposition.
A continuum version of this decomposition has been introduces in recent
work of Miller and Sheffield as part of a characterization of the Brownian
map \cite{MS15axiom}.

\subsection{Outline of proof}

A naive approach to proving recurrence of the HUIPT is to use the result of
Gurel-Gurevich and Nachmias in \cite{GN12}.  Let $B_r$ be the hull of the
combinatorial ball of radius $r$ around the root (the hull is obtained by
adding the finite components of the complement of the ball).  These are
finite planar graphs with exponential tail on the degrees, so their limit
is almost surely recurrent.  If the root is near the boundary, then the
limit is the HUIPT.  However, the root is unlikely to be near the boundary,
and the limit is the full plane UIPT.  If we can show that the limit contains
$H$ as a subgraph, then we would be done.  However, the limit is the UIPT,
and inclusion of the HUIPT is an open problem.

A more refined approach is to find some subset $S$ of the vertices of the ball 
such that if we pick uniformly a root uniformly from $S$ we obtain a limit
which contains $H$ as a subgraph.  One natural choice is to set $S$ to be
$\partial B_r$, so that the limit is the HUIPT.  However, since $|\partial
B_r| \approx r^2$, this set is much smaller than the volume of $B_r$.  Thus
the limit is not absolutely continuous with respect to the weak local limit
of $B_r$, and we are still short of a proof.

An improvement would be to take $S$ to be the union of the boundaries
$\partial B_j$, for $1 \le j \le r$.  This set is still much smaller than the
volume of $B_r$.
However, the situation can be salvaged: This set $S$ disconnects the
balls into small components (the blocks below); Understanding the structure
of $S$ gives some control over the structure of the resulting limit.
One can circle pack the limiting graph, and the circles corresponding to
the set $S$ will have no accumulation points in the plane.  Moreover,
\cref{lem:delta_s} gives us control over the number of vertices of $S$ in a
Euclidean ball.
In practice, it is more convenient to replace $B_r$ by a different subgraph
of the HUIPT, which is done below.

In order to complete the proof, we also need some new estimates on
the volume of balls in the HUIPT under a certain modified metric, as well
as estimates on vertex degrees.  With these in place, we can push through
the proof of \cite{GN12}.

We comment that there are also natural measures on half planar
quadrangulations, and more general `uniform' half planar maps.  There seems
to be no crucial obstacle to extending our results to such more general
classes of maps.  We restrict here to triangulations where the the layer
decomposition is particularly nice.  As noted, a similar decomposition was
used by Krikun for quadrangulations, and with care it seems the layered
structure as well as the rest of our argument can be carried over to such
more general maps.

\paragraph{Organization.}

In \cref{sec:background} we include some background material which we use,
concerning the weak local topology, planar maps, the UIPT and HUIPT, and
circle packings.  Readers familiar with these topics may wish to skip to
\cref{sec:layer} where we describe the layer decomposition of the HUIPT,
and describe the full plane map $M$ containing the HUIPT.  We also prove
there estimates on the volume growth and vertex degrees in $M$.
In \cref{sec:convergence} we show that a certain sequence of finite maps with
suitable distribution for the root converge to $M$.
Finally, in \cref{sec:recurrence} we combine all ingredients and prove
\cref{thm:main}.
We end with some comments on possible extensions and open questions in
\cref{sec:extens-open-quest}.

%%%%%%%%%%%%%%%%%%%%%%%%%%%%%%%%%%%%%%%%%%%%%%%%%%%%%%%%%%%%%%%%%%
\section{Background}
\label{sec:background}

\subsection{Planar maps: The UIPT and relatives}\label{S:maps}

Recall that a \textbf{planar map} is a proper embedding in the plane of a
connected (multi) graph in the plane, considered up to orientation
preserving homeomorphisms.  Components of the complement of the map are
called faces, and are assumed to be simple discs.  All our maps are
\textbf{rooted}, meaning there is a marked directed edge, called the root.
Equivalently, a planar map is a graph together with a cyclic order on the
edges at each vertex, such that the graph can be embedded with the edges
leaving the vertex in order.

Our maps will have a distinguished face which we shall call the external
face. The edges and vertices incident to the external face will be called
the \textbf{boundary} of the map.  When a map has a boundary, we shall
often assume the root is one of the boundary edges.  The boundary
throughout this paper will be either a simple cycle or a simple doubly
infinite path.  In the latter case, the map may be embedded in the half
plane with the boundary along a line.  Such a map is referred to as a half
plane map.

The local topology on the space of rooted graphs is generated by the
following metric: for rooted graphs $G,H$, we define
\[
d(G,H) = e^{-R} \qquad \text{where} \qquad
R = \sup\{r: B_r(G) \cong B_r(H)\}.
\]
Here $B_r$ denotes the ball of radius $r$ around the corresponding roots in
the graph distance, and $\cong$ denotes isomorphism of rooted maps.  For
maps, we require the equivalence relation to preserve the cyclic order on
edges at vertices.

This topology on graphs or maps induces a weak topology on the space of
measures on graphs (resp.\ maps).  A finite, possibly random, graph yields
a measure on rooted graphs by taking the root to be a uniform directed edge
(or vertex).  The \textbf{weak local limit} (or \textbf{Benjamini-Schramm
  limit}) of a sequence of finite graphs is the weak limit of the induced
measures.  The starting point of our work is the following result of
Gurel-Gurevich and Nachmias (and of Benjamini and Schramm with a bounded
degree assumption).

\begin{thm}[\cite{BeSc,GN12}]
  Let $G_n$ be finite planar graphs such that the degree of a uniform
  vertex has uniformly exponential tail.  Then $\lim G_n$ is almost surely
  recurrent.
\end{thm}

It has been known for some time \cite{UIPT1,UIPT2,AR13} that the uniform
measures on finite planar triangulations with boundary converge in the weak
local topology as the area of the map and the boundary length tend to
infinity.

\begin{thm}\cite{UIPT1,AR13}
  If $T_n$ is a uniform rooted triangulation (map with all faces triangles)
  with $n$ vertices, then the limit $T = \lim T_n$ exists. 
  If $T_{n,m}$ is a uniform boundary rooted triangulation with $m$ boundary
  vertices and $n$ internal vertices, then we have the limit
  \[
  T_{n,m} \xrightarrow[m,n/m \to \infty]{d} H.
  \]
\end{thm}

The limits $T$ and $H$ are the \textbf{UIPT} and \textbf{half plane UIPT}.
We denote the law of $H$ by $\H$.

The map $H$ also enjoys \textbf{translation invariance} with respect
to the root. This means that the law of the map remains invariant if
we translate the root along the boundary. See \cite{AR13} for a
detailed definition.

The distribution of a neighbourhood of the root in the HUIPT has a simple
and explicit formula which can be taken as an alternative direct definition
of HUIPT.

\begin{lemma}[\cite{AR13}]\label{lem:prob}
  Let $Q$ be a simply connected triangulation with a simple boundary, with
  some marked connected segment of $\partial Q$ containing the root, and let
  $H$ be the HUIPT.  Consider the event $A_Q$ that $Q$ is a sub-map of $H$
  with the roots coinciding and the marked segment being the 
  intersection of $Q$ with $\partial H$. Then
  \[
  \H(A_Q) = 6^{\#V_i(Q)} 9^{-\#F(Q)}
  \]
  where $\#V_i(Q)$ is the number of vertices of $Q$ not in $\partial H$ and
  $\#F(Q)$ is the number of faces of $Q$.
  Moreover, conditioned on $A_Q$, the complement $H\setminus Q$ also has
  law $\H$.
\end{lemma}

The final claim of this lemma is referred to as the \textbf{domain Markov
  property} of the HUIPT (see \cite{AR13}).

\subsection{Peeling} \label{sec:peeling}

One of the main tools we are going to use is known as \textbf{peeling}
which was introduced by Watabiki \cite{Wat} and given its present form by
Angel \cite{UIPT2}.  This technique can be applied to more general class of
maps, we focus primarily on HUIPT. The central idea is to explore (or
``peel'') a map face by face.  There can be many possible algorithms to do
it, and generally an algorithm is chosen depending on the purpose.  The
domain Markov property in the HUIPT gives the peeling process a rather
simple form.  For further applications of this powerful tool see e.g.\
\cite{AC13,curien13glimpse,menard2013perc,AR13,BC11}.

Consider the unique triangle incident to the root edge of the half plane
UIPT $H$.  One of the following two events must occur: With probability
$2/3$, the triangle can be incident to an internal vertex.  Otherwise the
triangle incident to the root edge is attached to a vertex on the boundary
which is at a distance $i$ to the left (resp.\ right) of the root edge
along the boundary.  Let $p_i$ be the probability of this event.  Moreover,
let $p_{i,k}$ be the event that the finite face enclosed by such a triangle
has $k$ vertices.  Let $\phi_{k,i}$ denote the number of triangulations of
an $i$-gon with $k$ internal vertices.  The following were derived in
\cite{UIPT2}.
\begin{equation}\label{eq:crit_pik}
  \begin{split}
    p_{i,k} &= \phi_{k,i+1} \left(\frac{1}{9}\right)^i \left(\frac{2}{27}
    \right)^k \\
    p_i &= \sum_kp_{i,k}=\frac{2}{4^i}
    \frac{(2i-2)!}{(i-1)!(i+1)!} \sim \frac1{2\sqrt{\pi}}i^{-5/2}
  \end{split}
\end{equation} 

The \textbf{Boltzmann triangulation} of an $m$-gon with weight
$q\leq\frac{2}{27}$, is the probability measure on that assigns weight $q^n
/ Z_m(q)$ to each rooted triangulation of the $m$-gon having $n$ internal
vertices, where
\[
Z_m(q) = \sum_n \phi_{n,m} q^n .
\]
The partition function $Z_m$ can be computed explicitly, and is finite for
$|q|\leq 2/27$.  When peeling a face, on the event that the face
connects to a vertex at distance $i$, the resulting component with boundary
$i+1$ is filled with a Boltzmann triangulation with weight $2/27$.

Having revealed the triangle incident to the root edge and the finite
component of its complement (if any), the unrevealed map is another half
plane map having law $\H$ by the domain Markov property.  This enables us
to peel the HUIPT via a succession of i.i.d.\ peeling steps.  Note that the
probabilities $p_{i,k}$ do not depend on the edge we choose to peel, by
translation invariance.

\subsection{Circle Packings}

As in some prior works \cite{BeSc,GN12,AHNR14}, circle packings play a
central role for us. We state here the two key results needed.  We refer
the reader to \cite{K36} and the above papers for further information.

A \textbf{circle packing} of a graph $G$ is a collection of circles in the
plane with disjoint interiors, one corresponding to each vertex, such that
two circles are tangent if and only if the corresponding vertices are
adjacent.  The Kobe-Andreev-Thurston Circle Packing Theorem states that
every finite planar graph has a circle packing.  There are extensions to
infinite planar triangulations, which we do not need at present.

In order to control the geometry of graphs in terms of circle packings, it
is useful to control the ratio of radii of circles.  This is done by the so
called Ring Lemma, which states that in a circle packing of a
triangulation, the ratio of radii of adjacent circles is bounded by some
constant depending only on the maximal degree of the graphs (for
non-boundary vertices).

%%%%%%%%%%%%%%%%%%%%%%%%%%%%%%%%%%%%%%%%%%%%%%%%%%%%%%%%%%%%%%%%%%
\section{Layer decomposition}\label{sec:layer}

Given a half planar map $H$, we define its layer decomposition as follows.
For each $i$, we will have a half plane map $H_i$.  These will form a
decreasing family of sub-maps of $H$, and each is a half plane map.  The
boundary of $H_i$ is denoted $S_i$, and  is a doubly infinite simple path
in $H$.  The vertices in $S_i$ are called \textbf{skeleton vertices}.

Inductively, we start with $H_0=H$ and its boundary $S_0=\partial H_0$.
Having defined $H_i$ and $S_i$, define the \textbf{layer} $L_{i+1}$ to be the
hull (relative to $H_i$) of the set of faces of $H_i$ incident to the
boundary $S_i$.  Thus $L_{i+1}$ forms a layer near the boundary of $H_i$.  We
then define the next sub-map $H_{i+1} = H_i \setminus L_{i+1}$, and $S_{i+1}$
to be its boundary (see \cref{fig:one_layer}).
For each $i$ we have that $S_i$ is a simple infinite path which
separates $L_{i}$ from $L_{i+1}$.  Conversely, the boundary of $L_i$ is $S_i
\cup S_{i+1}$.  Note also that by construction the sets $S_i$ are disjoint.

\begin{figure}
  \centering
  \includegraphics[width=.8\textwidth]{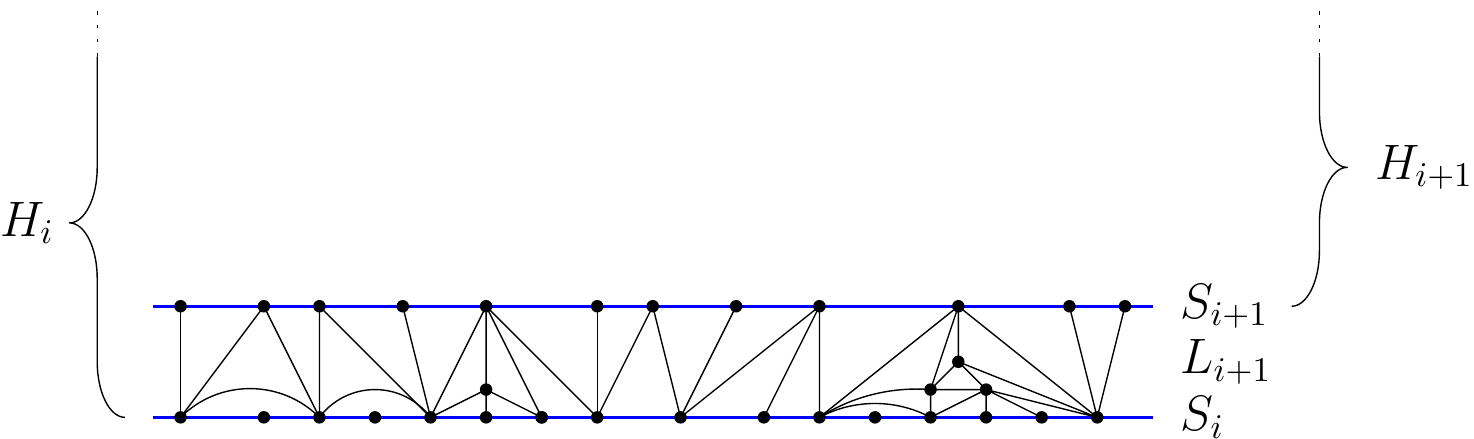}
  \caption{Construction of layer $L_i$ from $H_i$.  $L_i$ is the hull of
    all faces incident to the boundary $S_i$.  The entire HUIPT is $H_0$.}
  \label{fig:one_layer}
\end{figure}

Note that we have not yet determined a root for the maps $H_i$.  A root can
be chosen for each $H_i$ in various manners, and we will do that below.
However, the construction above is independent of the choice of root.

Let $e$ be some edge in $S_{i+1}$ for some $i\ge 0$.  Then there is unique
face in $L_{i+1}$ containing $e$, and the third vertex of that face must be in
$S_i$, since otherwise that face would not have been included in $L_i$.
For two adjacent edges $e,e'\in S_{i+1}$, the corresponding triangles of
$L_i$ that contain $e$ and $e'$
split $L_i$ into two infinite and one finite component.  We refer to the
finite components arising in this way as \textbf{holes}, since the sub-map
induced by skeleton vertices is missing all vertices in such holes.  Note
that it is 
possible that the two triangles share a common edge, in which case the hole
degenerates to that single edge.  It is also possible that the two
triangles share two vertices, one in $S_i$ and one in $S_{i+1}$, but the
edges are distinct, and in that case the hole is a $2$-gon.  Both occur in
the lower layer in \cref{fig:blocks}.
This observation implies that $L_i$ can be
decomposed as an alternating sequence of (possibly degenerate) holes and
faces containing the edges of $S_{i}$.  We can thus partition $L_i$ to a 
sequence of \textbf{blocks}, where each block consists of a hole and the
triangle immediately to its right.  The \textbf{lower
  boundary} of a block in $L_i$ is the set of edges of $S_{i-1}$ in the block,
which can be any non-negative integer.  (The upper boundary always consists
of a single edge.)  Apart from the lower and the upper boundary, the
block has two more boundary edges, where it is attached to blocks to its left 
and right.

\begin{figure}
  \centering
  \includegraphics[width=.8\textwidth]{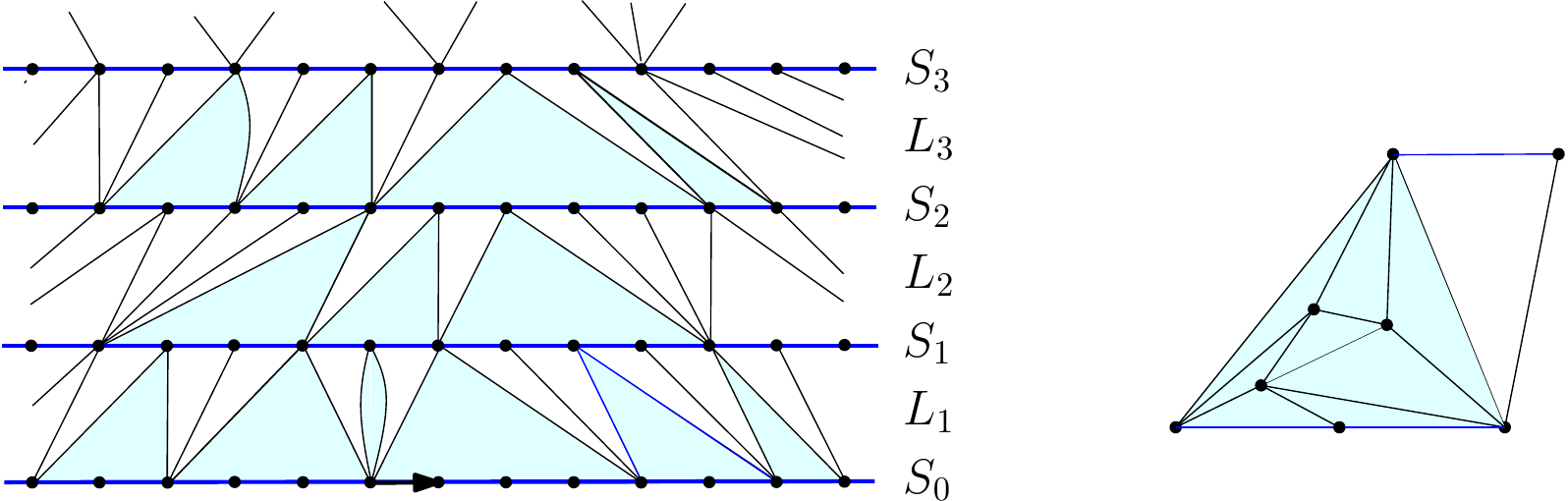}
  \caption{Left: Decomposition of layers into alternating holes and faces
    adjacent to top boundary edges.  Holes are shaded, and vertices and
    edges within holes are not shown. Note that some holes degenerate.
    Right: A hole and face to its right form a block.}
  \label{fig:blocks}
\end{figure}

\subsection{Decomposition of the half plane UIPT}
\label{sec:decomp-half-plane}

Up to this point we described the layer decomposition of a general half
plane map.  We now focus our attention on the specific case of the
half-plane UIPT.  While in our case, the description above is faithful, in
arbitrary half-plane maps things could break down.  For example, it is
possible that $L_1$ is the entire map. Indeed, this is the
case in the sub-critical half plane maps with the domain Markov property
that were constructed in \cite{AR13}.

\begin{lemma}
  \label{L:H_i_law}
  For the HUIPT, almost surely, $L_1$ is not the entire half plane, $S_1$
  is a doubly infinite simple path and $H_1$ is also a half plane map.
  Moreover, if we choose a root for $H_{i+1}$ as a function of
  $L_1,\dots,L_i$, then $H_{i+1}$ has the law of the half plane UIPT, and
  is independent of $L_1$.  Consequently, the layers $\{L_i\}_{i \ge 1}$ are
  i.i.d.
\end{lemma}

\paragraph{Peeling to reveal a layer.}
To prove \cref{L:H_i_law}, it shall be useful to consider the following 
application of peeling in the half plane UIPT.  
An analogue of this for the UIPT was used in \cite{UIPT2} to study the volume
growth of the UIPT.  In the HUIPT, the process becomes simpler.
Initially, make the root edge active. At any later time, the active edges
are those at the boundary of the unseen part of the map that are not on the
original boundary.  The active edges form a single contiguous segment, and
we peel either the rightmost or leftmost active edge. Let $Y_n$ be the
number of active edges in this segment after $n$ steps (with $Y_0=1$).

Let $Y_n$ be the length of this segment after $n$ steps, except
that by convention we set $Y_0=1$.
Define now the i.i.d.\ sequence $\xi_n$ as follows.  If the $n$th step
connects the peeled edge to a new internal vertex, then $\xi_n=1$.  If it
connects to a vertex at distance $i$ towards the rest of the active segment
then $\xi_n=-i$.  Finally, if the face connects to a vertex to the right,
$\xi_n=0$.  It is easy to see that $\xi_n$ determines the change in $Y_n$.
Specifically, he have
\begin{equation}
 Y_n = (Y_{n-1} + \xi_n) \vee 1.\label{eq:Yn}
\end{equation}

The $\xi_n$ variables are i.i.d.\ with distribution
\begin{equation}\label{eq:xi}
  \P(\xi=i) = \begin{cases}
    2/3 & i=1, \\ 1/6 & i=0, \\ p_i/2 & i<0,
  \end{cases}
\end{equation}
where $p_i$ is given in \eqref{eq:crit_pik}.  It follows from the
computations in \cite{UIPT2} that $\E(\xi) = 1/3$.
Note also that every peeled face is incident to some vertex in the
original boundary, and so all faces revealed in this procedure are part of
$L_0$.  Finally, the number of edges of the original boundary that are
swallowed at each step are also i.i.d.\ with mean $1/3$.

\begin{proof}[Proof of \cref{L:H_i_law}]
  When peeling to reveal a layer, since $\E\xi=1/3>0$, the strong law of
  large numbers implies that $Y_n/n$ converges to $1/3$ almost surely and in  
  particular, $Y_n$ tends to infinity almost surely.  

  Start by peeling at the rightmost active edge $n$ times.  The law of
  large numbers ensures that the number of edges to the right of the root
  that are swallowed grows like $n/3$.  While some of the previously active
  edges contributing to $Y_n$ are swallowed at a later step, at each $n$
  there is probability $1/3$ that $Y_n = \inf_{t\geq n} Y_t$ (via Theorem~3
  of \cite{AR08}), and in that case, only the rightmost active edge is
  subsequently swallowed.  In particular, the number of boundary vertices
  to the left of the root that are swallowed is tight.

  Next, reverse direction, and peel towards the left for $n$ additional
  steps.  At this time we revealed some finite map $P_n$ which contains all
  faces incident to edges within distance $a_n$ along the boundary to the
  left and $b_n$ along the boundary to the right with both $a_n,b_n$ close
  to $n/3$ with high probability.  Thus $P_n$ converges to the layer $L_1$.
  Moreover, the number of edges from the $Y_n$ that are swallowed in the
  second stage is tight, and therefore with high probability some of them
  remain on the boundary as $n\to\infty$. This implies that $L_1$ is not
  the entire map.

  To see that $H_1$ is again a half plane UIPT, and is independent of
  $L_1$, root $M_n = H\setminus P_n$ at some canonically chosen vertex
  $\rho_n$, say the first one reached in the process that is on the
  boundary of $P_n$.  From the domain Markov property, $(M_n,\rho_n)$ has
  the law of the half plane UIPT, and is independent of $P_n$.  This
  completes the proof, since $\rho_n$ is eventually constant, and so
  $(M_n,\rho_n)$ converges to $ (H_1,\rho)$.  Finally, by translation
  invariance of $H_1$, we can choose a root for $H_1$ as any function of
  $L_1$ and the law of $H_1$ will not change.  

  By induction, the same holds for subsequent layers.
\end{proof}

\begin{prop}
  \label{prop:L0}
  In each layer $L_i$ we have the following.
  \begin{enumerate}[nosep]
  \item The blocks are independent.  All have the same law, except for the
    block containing the root edge which is biased by the size of its lower
    boundary. Given the block containing the root edge, the root edge is
    distributed uniformly among the edges in its lower boundary.
  \item The number of edges $B$ in the lower boundary of a block (other
    than the one containing the root edge) satisfies $\E(B) = 1$ and
    $\P(B>t) \sim ct^{-3/2}$.
  \item Conditioned on the lower boundary length of $B$, the component of
    $H$ within the hole is a Boltzmann map of an $(B+2)$-gon with parameter
    $2/27$.
  \end{enumerate}
\end{prop}

We remark that the proof yields the precise distribution of the lower
boundary size of a block in terms of the partition function of
triangulations, which is explicitly known.  We do not need the formula for
this distribution.

\begin{proof}
 We enumerate the blocks $\{B_i\}_{i \in \Z}$  using integers with $B_0$ 
being the block containing the root edge. Consider a sequence of blocks 
$(B_i)_{i\in[j,k]  \cap \Z}$ with $j\leq 0\leq k$.
  Suppose $B_i$ has $b_i$ lower boundary edges and $v_i$ vertices in its
  hole.  Let $B_0$ also have a marked edge on its lower boundary.  We
  compute the probability that these are consecutive blocks of $L_1$, with
  the marked edge of $B_0$ being the root edge.  A block has $2v_i + b_i +
  1$ faces.  Joining these blocks, the total number of vertices internal to
  $M$ is $V = 1+\sum v_i+1$, including also the upper boundary vertices.
  Letting $F=\sum 2v_i+b_i+1$, by \cref{lem:prob}, the probability of these
  blocks being part of the map is $6^V 9^{-F}$.

  In order for these to be blocks in $L_1$, it is also necessary that if we
  peel along the boundary to the right than no internal vertex (revealed so 
far) is swallowed,
  and that to the left the first step reveals an internal vertex, and
  afterwards no additional internal vertex is swallowed.  These have
  probability $1/3$ and $2/27$ respectively, which are just a constant for
  us.  Thus the probability of having the blocks $B_i$ is
  \[
  C \prod 6^{1+v_i} 9^{-2v_i-b_i-1} = 
  C \prod \frac23 \left(\frac{2}{27}\right)^{v_i} 9^{-b_i}.
  \]
  for some absolute constant $C$.  (Careful calculation shows $C=1$;
  However, we need not worry about the value of
  $C$, since it is  determined by the fact that these probabilities add up
  to $1$, and its value is canceled out in what follows.)
  This shows that the blocks are independent, that given $b_i$, and that
  the hole is filled with a Boltzmann triangulation with parameter $2/27$.
  Moreover, the probability that $b_i=m$ is proportional to $\sum_n
  \phi_{n,m} (2/27)^n 9^{-m}$, which decays as $ct^{-5/2}$ via 
\eqref{eq:crit_pik}.  Finally, for
  $B_0$ there is a marked edge on the lower boundary, so the probability of
  it having $b_0=m$ is proportional to $\sum_n m \phi_{n,m} (2/27)^n
  9^{-m}$, i.e.\ it is a biased by its lower boundary. That the root is
distributed uniformly among the vertices in the lower boundary of its
block follows from translation invariance.
\end{proof}

We are going to denote by $\B$ the \textbf{law of a block} described in
\cref{prop:L0} and by $\Bb$ the law of the block containing the root (i.e.\
biased by the size of its lower boundary).  Let $D_{\max}$ denote the
maximal degree of 
a vertex in a block.  The following shall come as no
surprise to the reader familiar with random maps.

\begin{lemma}\label{lem:Dmax}
  Then for some $c,C>0$ and all $r \ge 1$, $\B(D_{\max} > r)
  \le C e^{-cr}$, and $\Bb(D_{\max} > r) \le C e^{-cr}$.
\end{lemma}

The proof follows a fairly standard argument, and is not too difficult,
however, we have not been able to locate this statement in the literature.
The proof is separated
into three steps.  The first is known lemma about the degree of a boundary
vertex in a Boltzmann triangulation.

\begin{lemma}\label{lem:deg_boundary}
  There are constants $c,C$ such that for any $m$, if $B$ is a Boltzmann
  triangulation of an $m$-gon, rooted at $\rho\in\partial B$, then
  $\P(d_\rho > r) \leq C e^{-cr}$.
\end{lemma}

\begin{proof}
  This follows from the same argument for exponential degree
  distribution in the UIPT from \cite{UIPT1}.
\end{proof}

\begin{lemma}\label{lem:deg_given_boundary}
  There are constants $c,C$ such that for any $m$, if $B$ is a Boltzmann
  triangulation of an $m$-gon, rooted at $\rho \in \partial B$, and
  $D_{\max}$ the maximal degree of any internal vertex, then
  $\P(D_{\max} > r) \le C m^2 e^{-cr}$.
\end{lemma}

\begin{proof}
  We perform a peeling process to reveal $B$, each time peeling at some
  edge and revealing one face.  However, if the revealed face separates the
  map into two sub-maps, we do not reveal either of them immediately, but
  proceed to explore one and then the other in some arbitrary order.  Thus
  at time $i$ we have revealed $i$ faces, and the remainder of $B$ is a
  collection of independent Boltzmann maps of some cycles (unless the
  process has terminated, in which case there is no complement).

  Let $\mathcal F_i$ be the sigma algebra generated by this peeling process
  up to the $i$th step.  Let $A_i$ be the event that a new vertex is
  revealed at step $i$, and let $A_{i,r}\subset A_i$ be the event that this
  vertex has degree greater than $r$.  Our goal is to bound $\P(\cup_i
  A_{i,r})$:
  \begin{align*}
    \P(D_{\max} > r) & \le \sum_i \P(A_{i,r}) \\
    & \le \sum_i \E \left[ \P(A_{i,r} | \cF_i) \right] 
    \le \sum_i \E \left[ \mathbbm{1}_{A_i} \P(d_i>r | \cF_i) \right]\nonumber
  \end{align*}
  since $A_i\in\cF_i$, where $d_i$ is the degree of the vertex revealed at
  step $i$.  When a new vertex is revealed, its degree is $2$.  Conditioned
  on $\cF_i$, the component of vertex $i$ is filled with a Boltzmann map,
  and so by \cref{lem:deg_boundary}, we have $\P(d_i>r | \cF_i) \leq C
  e^{-cr}$.  Thus we have
  \begin{align*}
    \P(D_{\max} > r)
    & \le \sum_i \E \mathbbm{1}_{A_i} C e^{-cr} \\
    & = C e^{-cr} \E |B|  \le  C m^2 e^{-cr},
  \end{align*}
  Where $|B|$ is the number of vertices in $B$, which is known to have
  expectation of order $m^2$ (see \cite[Proposition 5.1]{UIPT1}).
\end{proof}

\begin{proof}[Proof of \cref{lem:Dmax}]
  We know from the second item in \cref{prop:L0} that the probability that
  the boundary of a Boltzmann map of law $\B$ or $\Bb$ is larger than
  $e^{\eps r}$ is exponentially small.  The rest follows from
  \cref{lem:deg_given_boundary} for $\eps=c/3$, with the $c$ from
  \cref{lem:deg_given_boundary}.
\end{proof}

\subsection{Full plane extension of $H$}\label{sec:full-plane-extension}

Given the layer decomposition of the half plane UIPT, we now construct a
plane triangulation $M$ with no boundary which contains $H$ as a sub-map.
Eventually we will also show that $M$ is almost surely recurrent, implying
\cref{thm:main}.

To construct $M$ start with the half plane UIPT $H = \bigcup_{i\geq1} L_i$.
We add a sequence of layers below the boundary $S_0$, to create a full plane map.
For each $i\le 0$, the layer
$L_i$ is composed of a doubly infinite sequence of i.i.d.\ blocks with law
$\B$, attached to form a layer.  Note that there is no size biased block in
$L_i$ for $i\le 0$.  Thus for $i\leq 0$, if $L_i$ is rooted at some vertex
on the top 
boundary, it is translation invariant in law.  We then identify the top
boundary of $L_i$ with the bottom boundary of $L_{i+1}$ for every
$i\le 0$.  The full plane map is $M = \bigcup_{i\in\Z} L_i$.
By translation invariance of the lower layers, the law of the resulting
full plane map does not depend on which edge in the bottom boundary of
$L_{i+1}$ is identified with the root of $L_i$.
The boundary between $L_{i+1}$ and $L_{i}$ is denoted $S_i$ also for $i<0$.
Vertices of $S_i$ for any $i$ are also called skeleton vertices.
The map $M$ is rooted at the
root $\rho$ of $H$.
\cref{thm:main} is an immediate consequence of the following.

\begin{thm}\label{thm:M_rec}
  The full plane extension $M$ is almost surely recurrent.
\end{thm}

Let $S = \bigcup S_i$ be the skeleton of the map $(M,\rho)$.  We define a
graph on $S$, where two vertices $x,y\in S$ are adjacent if they are both
incident to some hole in some layer of $M$.  Call this graph $\Sk(M)$.
Note that adjacent vertices in any $S_i$ are adjacent in $\Sk(M)$, and that
neighbours in $\Sk(M)$ are either both in $S_i$ for some $i$ or in $S_i$ and
$S_{i+1}$ for some $i$.  Note that $\Sk(M)$ does not have a natural map
structure. However, it consists of finite cliques, and the intersection
graph between the cliques is planar and approximates $M$ in some ways.  (We
do not need this, and do not make this precise here.) The nonempty
blocks in $M$ corresponds naturally to blocks in $\Sk(M)$ with the
vertices in the boundary of a block forming a clique (we will keep
referring to them as blocks in $\Sk(M)$) We shall use the
notation $\Sk(A)$ to denote the corresponding graph also for various
$A \subset M$, which will be the subgraph of $\Sk(M)$ induced by vertices
of $A$.

Our immediate goal is to show that $\Sk(M)$ has polynomial volume growth.
Let $B_{sk}(r)$ denote the combinatorial ball of radius $r$ around
$\rho$ of $\Sk(M)$ along with all the finite components of its complement.

\begin{prop}\label{P:volume_skel}
  The random variables $r^{-4} | B_{sk}(r) |$ form a tight family.
\end{prop}

We start with a few additional definitions.  For any skeleton vertex $v \in
S_i$ we can associate a unique hole in the layer above it ($L_{i+1}$) that
contains the edge of 
$S_i$ to the right of $v$.  This hole is also incident to $S_{i+1}$ at a
unique vertex.  Call this vertex the \textbf{parent} $p(v)$ of $v$, and
define $p^{(r)}$ to be the $r$ fold composition of the operation $p$.
Equivalently, the parent of a vertex $v\in S_i$ is the rightmost vertex of
$S_{i+1}$ that is adjacent to $v$ in $\Sk(M)$.

Using \cref{prop:L0}, it is natural to study the maps via certain critical
Galton-Watson trees derived from the block decomposition.  A similar
construction was used by Krikun for the full plane UIPQ in
\cite{krikun2005local,kriUIPT}, except that the trees there are not
Galton-Watson trees.  In the layered map, we define a tree as follows.  The
vertices are the skeleton vertices.  The parent of $v$ is $p(v)$.  The set
of all offspring of a vertex $v$ form a tree, which we denote by $T_v$.
The following is clear from this discussion and \cref{prop:L0}.

\begin{lemma}\label{lem:Kri_tree}
  For every skeleton vertex $v\in S_0$, the tree $T_v$ is a critical Galton-Watson
  tree with offspring distribution $Z$ satisfying $\P(Z>k) \sim ck^{-3/2}$.
  Moreover, the law of the (infinite) tree rooted at $\rho$ is the fringe
  of the same Galton-Watson tree.
\end{lemma}

See Aldous \cite{fringe91} for the theory of fringes of trees.  We do not
need this theory in full generality, but use the following consequences:
\begin{itemize}[nosep]
\item Every ancestor $p^{(r)}(\rho)$ of $\rho$ has offspring distributed
  as size biased $Z$.
\item The previous ancestor $p^{(r-1)}(\rho)$ is a uniform child of its
  parent.
\item All other children of  $p^{(r)}(\rho)$ produce independent
  Galton-Watson trees of offspring.
\end{itemize}

While a-priori it is not obvious that our parent definition defines a
single tree and not a forest.  The connectivity can be deduced from the
criticality of the trees by showing that for any two vertices of $x,y \in
S_i$, $p^{(r)}(x) =p^{(r)}(y)$ for $r$ large enough.  This is
straightforward, but we do not need the connectivity for our purposes, so
omit details.  \cref{P:volume_skel} is an immediate consequence of the
following.

\begin{lemma}\label{lem:vol_from_top}
  Let $\widetilde B(2r)$ be the hull of the ball of radius $2r$ around
  $p^{(r)}(\rho)$ in $\Sk(M \setminus H_r)$, then $r^{-4} |\widetilde
  B(2r)|$ are tight.
\end{lemma}

Note that $M \setminus H_r$ is the half plane map consisting of $L_i$ for
$i\leq r$, and thus we are considering a ball in the skeleton of this map,
centred at a boundary vertex.

\begin{proof}[Proof of \cref{P:volume_skel}]
  The ball $B_{sk}(r)$ is contained in layers $L_{-r},\dots,L_r$, since the
  $\Sk$-distance from $\rho$ to $p^{(r)}(\rho)$ is $r$, we have that
  $B_{sk}(r) \subset \widetilde B(2r)$.
\end{proof}

For the proof of \cref{lem:vol_from_top}, we require the following standard
estimate regarding survival probabilities for Galton-Watson trees.  Note
that the offspring distribution in $T_v$ has infinite second moment, so the
survival probabilities do not decay as $c/n$ and the volume up to level
$n$, conditioned on survival to that level is not quadratic.  Recall that
it is possible to obtain an infinite version of a critical Galton-Watson
tree by conditioning it to survive up to generation $n$ and then taking the
limit as $n \to \infty$ in the local weak topology. Moreover such an
infinite tree has a single infinite path -- the spine -- and finite trees
attached to it.  The tree can be described as follows. The root has an
offspring distribution which is the size biased version of the original
offspring distribution. A uniformly picked child $v$ has a tree conditioned
to survive below it, while all its siblings have unconditioned
Galton-Watson trees of descendants.  We refer to \cite{kestensubdiff} for a
detailed account of Galton-Watson trees conditioned to survive.  The two
following lemmas are standard.  Proofs can be found e.g.\ in \cite{KC08}.

\begin{lemma}\label{L:GW_height}
  Let $T$ be a critical Galton-Watson tree with offspring distribution
  satisfying $\P(Z>k) \sim c k^{-3/2}$. Then the probability that $T$
  survives to generation $n$ decays as $cn^{-2}$ for some $c$. 
\end{lemma}

\begin{lemma}
  \label{L:GW_vol}
  Let $T^*$ be a critical Galton-Watson tree with offspring distribution
  satisfying $\P(Z>k) \sim c k^{-3/2}$ conditioned to survive.  Let $W_n$
  be the number of offspring in the $n$th generation of $T^*$.  Let $Y_n =
  \sum_{t=1}^n W_n$. Then $n^{-2}W_n$ and $n^{-3}Y_n$ converge to some
  non-zero random variables in law.
\end{lemma}

\begin{proof}[Proof of \cref{lem:vol_from_top}]
  We will construct inductively a growing sequence of subgraphs $\{P_i\}$
  around $p^{(r)}(\rho)$ in such a way that $P_{i} $ contains the hull of
  the ball of radius $i$ around $p^{(r)}(\rho)$ in $\Sk(M \setminus H_r)$.
  For all $i$, all vertices of $P_i$ will be in layers $S_{r-i},\dots,S_i$
  (as is the ball they bound).
  As a basis, we set $P_0 = \{p^{(r)}(\rho)\}$.  We also define a two-sided
  sequence of vertices in $S_r$, starting with $U_0 = \{p^{(r)}(\rho)\}$.
  For $i>0$, having defined $P_{i-1}$, and $U_{1-i},\dots,U_{i-1}$, let
  $U_i$ be the nearest vertex in $S_r$ to the right of $U_i$ such that the
  tree below $U_i$ survives for at least $i$ generations.  Similarly,
  $U_{-i}$ is the nearest vertex in $S_r$ to the left of $U_{1-i}$ such
  that the tree below $U_{-i}$ survives for at least $i$ generations.

  We now define $P_i$ as follows. We take all vertices of the trees at
  $U_{-i}$ and at $U_i$, from $S_r$ down to $S_{r-i}$.  Since the
  definition of the trees is asymmetric in that of the holes, it is
  convenient for the tree at $U_i$ to also take the rightmost vertex of the
  rightmost hole at each level.  (This vertex is in a tree further to the
  right). Finally, we also take in each of these levels all vertices
  between these two trees.

  Note that since the first $i$ levels of the tree below $U_i$ are strictly
  to the right of the tree below $U_i$, and similarly on the left, we have
  that $P_i$ indeed form an increasing sequence of subgraphs.  See
  \cref{fig:vol_proof} for an illustration.

  \begin{figure}
    \centering
    \includegraphics[scale = 0.7]{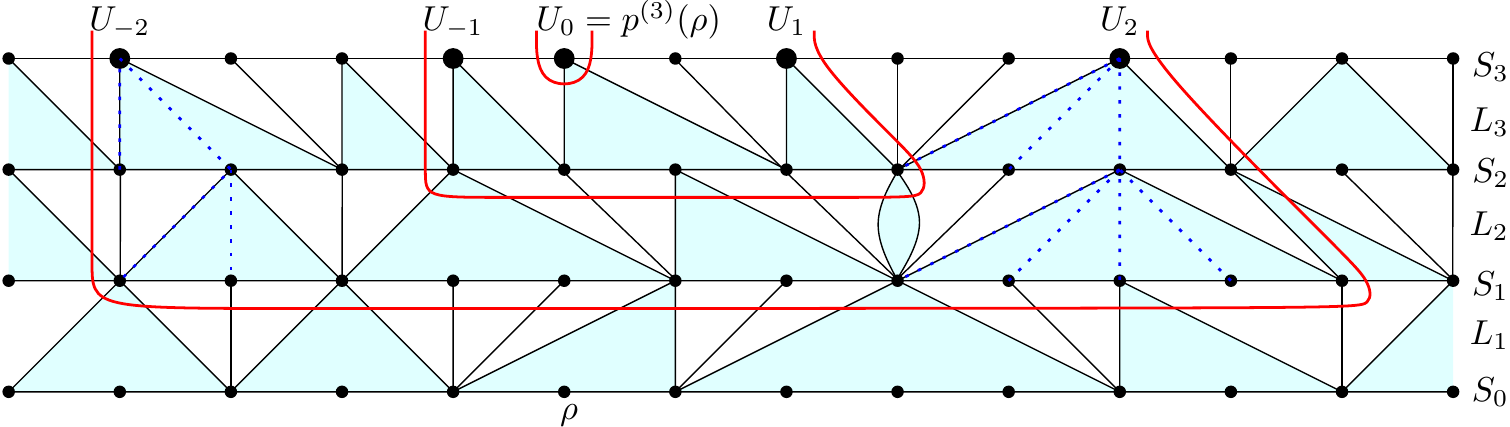}
    \caption{Construction of the maps $P_0$,$P_1$,$P_2$
       in the proof of \cref{lem:vol_from_top}. The vertices
       $U_{-2},\dots,U_2$ are larger.  The first two generations of the
       trees below $U_{\pm 2}$ are dotted in blue.  The boundaries of
       the maps are marked in red.  Each boundary
       forms a cutset around the previous one.}
    \label{fig:vol_proof}
  \end{figure}

  The rest of the proof consists of two claims.  First, that the set $P_i$
  contains the ball of radius $i$ around $p^{(r)}(\rho)$ in the map $\Sk(M
  \setminus H_r)$.  Secondly, we estimate of the size of these sets.

  The first claim is proved by induction.  Clearly the claim is true for
  $P_0$.  For the induction step, we argue that the internal boundary of
  $P_i$ (i.e.\ vertices of $P_i$ connected to its complement) is completely
  contained in the two trees below $U_{\pm i}$ together with the segment of
  $S_{r-i}$ between the two trees.  In particular, the boundary is disjoint
  of $P_{i-1}$, and hence each $P_i$ contains a ball of radius one around
  $P_{i-1}$.

  A level $S_j$ is naturally partitioned into intervals of vertices with a
  common parent.  The lower boundary of a hole is one such interval,
  together with the first vertex of the next interval to the right. Edges
  of $\Sk(M)$ are either within intervals, or between adjacent intervals,
  or between a vertex and its parent, or between a vertex and the parent of
  an adjacent interval.  Since every level from $S_{r-i},\dots,S_r$
  contains some vertices from the trees under $U_i$ and $U_{-i}$ these two
  trees indeed separate the rest of $P_i$ from vertices to the right and
  left.  Clearly only vertices in $S_{r-i}$ can be connected to vertices
  further down in the map, and the first claim is proved.

  Finally, we consider the size of $P_{2r}$, which consists of the first
  $2r$ generations from the trees rooted at each vertex between $U{_-2r}$
  and $U_{2r}$.  The tree rooted at $U_0$ is special: Its first $r$ levels
  are those of the tree conditioned to survive, with one vertex at each
  level having the size-biased offspring distribution.  After level $r$, it
  transitions to a critical tree.  The trees at all other vertices of $S_r$
  are critical Galton-Watson trees, except that the choice of $U_i$ is not
  independent of the trees.

  Fix some $\eps>0$.  For the tree at $U_0$, by \cref{L:GW_vol} we have for
  some $C$ that the number of vertices in generations $0,\dots,r$ is at
  most $Cr^3$ and the number of vertices at generation $r$ is at most
  $Cr^2$ with probability at least $1-\eps$.  Below each of the vertices at
  generation $r$ we consider the first $r$ generation of an independent
  critical Galton-Watson tree.  On the event that generation $r$ is not too
  large, this adds in expectation at most another $Cr^3$ vertices, and by
  Markov's inequality the total contribution from the tree at $U_0$ is at
  most $(C+C/\eps) r^3$ with probability at least $1-2\eps$.

  The trees at other vertices of $S_r$ are all independent critical
  Galton-Watson trees, and we consider the first $2r$ levels of these
  trees.  The expected size of each such tree is $2r+1$ (including its
  root).  It is convenient to identify the vertices of $S_r$ with $\Z$,
  with $U_0$ being $0$.  Between $U_{i-1}$ and $U_i$ we consider trees
  until finding one that survives to generation $i$, and so $U_i$ is a
  stopping time.  Since $U_i-U_{i-1}$ is geometric with mean of order
  $Ci^2$ (by \cref{L:GW_height}), we have $\E U_{2r} \leq Cr^3$.  By Wald's
  identity, the expected total size of the trees from $U_0$ to $U_{2r}$ is
  at most $C r^4$. By symmetry, the same holds for trees to the left of
  $U_0$, and the claimed tightness follows.
\end{proof}

\begin{lemma}\label{lem:DmaxM}
  Let $\Delta$ be the maximal degree in $M$ of the vertices in $B_{sk}(r)$.
  There exists $C>0$ such that $\P(\Delta > C \log r) \to 0$ as $r \to \infty$.
\end{lemma}

\begin{proof}
  This almost follows immediately from~\cref{P:volume_skel,lem:Dmax} except
  we need to take care of the fact that a vertex can be incident to many
  blocks in the layer above it. However, such a number is still Geometric,
  so the lemma follows. We make this rigorous below.

  For every edge $e$ in $(\bigcup_iS_i )\cap B_{sk}(r+1)$,
  consider the first edge to the right or left of this edge whose block has
  nonempty lower boundary. Since the blocks are independent and every
  block has a positive probability of having a non-zero lower
  boundary, the set of edges $S_e$ we need to check until we find such an
  edge has a Geometric number of elements. Using this information and \cref{lem:Dmax}, we see that the maximal
  degree $d_e$ of vertices in in all these blocks corresponding to
  edges in $S_e$ has exponential tail. It
  is easy to see that $\Delta \le 2\max_e d_e$ where the maximum is
  over all the edges in $(\cup_i S_i )\cap B_{sk}(r+1)$.
  Now using \cref{P:volume_skel}, we know that $\P(B_{sk}(r)>r^5)$
  converges to $0$. We can take a union bound over
  $r^5$ many edges on the event $\{B_{sk}(r)\le r^5\}$  and choose a large enough $C$ to arrive at
  the desired conclusion.
\end{proof}

%%%%%%%%%%%%%%%%%%%%%%%%%%%%%%%%%%%%%%%%%%%%%%%%%%%%%%%%%%%%%%%%%%
\section{$M$ as a distributional local limit}\label{sec:convergence}

We now define a sequence of finite maps $M_n \subset H$.  Each $M_n$
will inherit the layered structure from $H$, and so some vertices will be
skeleton vertices.  $M_n$ will have the property that if we select a root
$\rho_n$ uniformly from the skeleton vertices, then $(M_n,\rho_n)$
converges in distribution to $M$.

For a skeleton vertex $v\in H$, recall the definition of the parent $p(v)$
of $v$ from \cref{sec:full-plane-extension}, and that $p^{(k)}$ is the
$k$-fold composition of the operation $p$.  Now we set up a coordinate
system for skeleton vertices as follows.  The vertices of $S_k$ will have
coordinates $\{(k,n)\}_{n\in\Z}$, in the order they occur in $S_k$.  The
root vertex $\rho$ has coordinates $(0,0)$ and for any $k>0$, the vertex
$p^{(k)}(\rho)$ has coordinates $(k,0)$.  Having defined these, the vertex
of $S_k$ at a distance $j$ to the right (resp.\ left) of $(k,0)$ has
coordinates $(k,j)$ (resp.\ $(k,-j)$).  See \cref{fig:coordinate} for an
example.  Note that coordinates are only defined for $S_k$ with $k\geq 0$.

\begin{figure}[h]
  \centering
  \includegraphics[width=\textwidth]{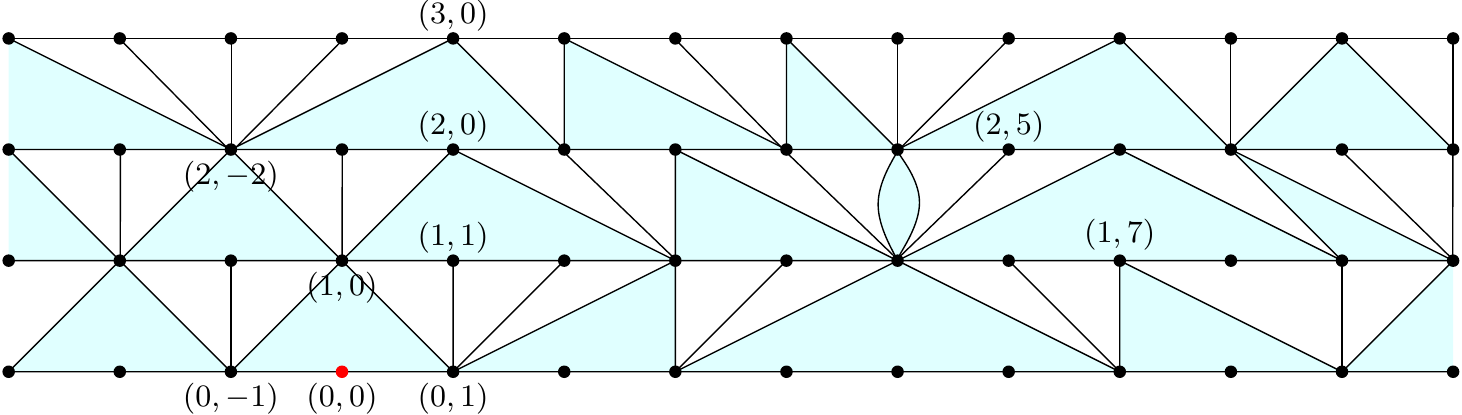}
  \caption{The coordinate system for $\Sk(M)$, with some coordinates noted.}
  \label{fig:coordinate}
\end{figure}

\begin{lemma}\label{lem:parent_child}
  Fix $k \ge 0$.  Define $\ell' = \ell'(k,\ell)$ by $(k+1,\ell') =
  p((k,\ell))$.  Then almost surely,
  \[
  \frac{\ell'}{\ell} \xrightarrow[\ell \to \infty]{} 1.
  \]
\end{lemma}

\begin{proof}
  The blocks in layer $L_{k+1}$ corresponding to vertices $(k+1,i)$ for
  $i>0$ form an i.i.d.\ sequence of blocks distributed as $\mathcal B$.  The
  statement is now an immediate consequence of \cref{prop:L0} (specifically
  that $\E B_i = 1$) and the Strong Law of Large Numbers which implies
  $\ell/\ell' \to 1$.
\end{proof}

Now define $M_n$ as follows.  The skeleton vertices of $M_n$, denoted
$\Sk(M_n)$ is the set $\{(i,j): 0 \le i \le n,1 \le j \le n\}$.  The holes
of $M_n$ includes all holes in $H$ all of whose skeleton vertices are
contained in $\Sk(M_n)$.
Finally take a root $\rho_n$ for $M_n$,
which is a uniformly selected selected skeleton vertex from $M_n$.

%% \note{why is the following lemma needed?}

%% \begin{lemma}\label{lem:local_properties}
%%   Let $B_{ij}$ denote the block in $L_i$ incident the vertices $(i,j)$ and
%%   $(i,j+1)$ of $S_i$.  Then the blocks $B_{ij}$ are independent, the blocks
%%   $\{B_{i0}\}_{i \ge 0}$ are have law $\Bb$ and the rest have law $\B$.
%%   Further, for any $i\ge 0$ and $j\in \Z$, conditioned on $L_1, \dots, L_i$, 
%% the map
%%   $H_{i}$ rooted at $(i,j)$ has law $\H$.
%% \end{lemma}

%% \begin{proof}
%%   The first assertion follows from \cref{prop:L0} while the second
%%   assertion follows from domain Markov property and translation invariance
%%   of $H$.
%% \end{proof}

Next, we show that $\rho_n$ is far away from the boundary of $M_n$ with
high probability.

\begin{lemma}\label{lem:root_inside}
  We have
  \[
  d_\Sk(\rho_n, \partial M_n) \xrightarrow[n,k\to\infty]{(d)} \infty
  \]
  where $d_\Sk$ is graph distance in $\Sk(M_n(k))$.
\end{lemma}

% \note{We could define the skeleton directly with the graph structure
%   below.  This is planar, and nicer in various ways, but the projection of
%   a path from $M'$ to $M$ is to the current skeleton graph.}

\begin{proof}
  Consider the subgraph of the skeleton graph where $x\sim y$ if they are
  either at the same level and adjacent, or one is the parent of the other.
  Let $d_p$ be the distance in this graph.  It is easy to see that
  $d_p(x,y) \leq 3d_\Sk(x,y)$, so it suffices to prove that for arbitrary
  $r$, with high probability $d_p(\rho_n, \partial M_n) \geq r$.

  Observe that since $M_n$ always has $n(n+1)$ skeleton vertices by
  definition, and the coordinates of $\rho_n$ are independent of $M_n$. Let
  the root $\rho_n$ have coordinates $(i,j)$.  As $n\to\infty$, with high
  probability $r < i < n-r$.  On this event, the ball $B_p(\rho_n,r)$ does
  not reach levels $S_0$ or $S_n$.  Fix any such $i$.  For any $\eps>0$
  there is some $M$ so that with probability at least $1-\eps$, for every
  vertex $(x,y)$ with $x\in[i-r,i+r]$ and $y>M$, if $(x',y')$ is adjacent
  to $(x,y)$ in the $d_p$ metric then $y'/y \in (e^{-\eps},e^{\eps})$.
  Call this event $B_i$, and assume it holds.  If $\rho_n=(i,j)$ and
  $j>e^{r\eps} M$ then every vertex in $B_p(\rho_n,r)$ has second
  coordinate in $[e^{-r\eps} j, e^{r\eps}j]$.  If $n$ is large enough, then
  with high probability $e^{r\eps} M < j < e^{-r\eps} n$, and then the ball
  $B_p(\rho_n,r)$ is contained in $M_n$.
\end{proof}

\begin{prop}\label{P:convergence}
  We have
  \[
  (M_n,\rho_n) \xrightarrow [n \to \infty]{(d)} (M,\rho)
  \]
  in the weak local topology.
\end{prop}

\begin{proof}
  The coordinates $(i_n,j_n)$ of a the uniform root $\rho_n$ tend to
  infinity in distribution as $n\to\infty$.  By \cref{lem:root_inside}
  large balls around $\rho_n$ are contained in $M_n$, so the weak local
  limit of $(M_n,\rho_n)$ is the same as the limit of $(M,\rho_n)$.  Since
  $i_n,j_n$ are independent of $M$, it suffices to show that for a fixed
  sequence $\{(i_n,j_n)\}\to\{\infty,\infty\}$, if we take $\rho_n =
  (i_n,j_n)$, then $(M,\rho_n)$ converges in distribution to the full plane
  map $(M,\rho)$.

  Given the layers $L_1,\dots,L_i$, the half plane map above them (denoted
  $H_{i}$) has law $\H$. Thus the layers above $\rho_n$ have the law of the
  HUIPT, and are independent of the $i$ layers below $\rho_n$.  Note that
  translation invariance implies that the block above $\rho_n$ is precisely
  size-biased, as are subsequent blocks above it.

  Since the blocks in the first $i$ layers are independent with law $\B$, the layer below $\rho_n$ has
  distribution similar to that of layer $L_0$ in $M$, except for one block
  at distance $j_n\to\infty$.  The same is true for all $i_n$ levels below
  $\rho_n$. By \cref{lem:root_inside}, the distance to these biased blocks
  tends to infinity, giving the result.
\end{proof}

%%%%%%%%%%%%%%%%%%%%%%%%%%%%%%%%%%%%%%%%%%%%%%%%%%%%%%%%%%%%%%%%%%
\section{Bounding degrees: the star-tree transformation}
\label{sec:surgery}

Following \cite{GN12}, we apply the so-called star-tree transformation to
our maps to get 
maps with bounded degrees.  These can then be embedded in the plane using
circle packings, which are better behaved when vertices have bounded
degrees.

The star-tree transform is constructed roughly as follows: starting with a
map $G$, possibly with large degrees, take its dual, which has large faces,
triangulate
each face to get a triangulation, and take the dual again to get a three
regular map $G'$ which is related to the original map.  The triangulation
step can be done in various ways, and we will be more specific below.  Each
vertex of $G$ of degree $d$ is replaced by a 3-regular tree which 
connects to other trees at its leaves.  Crucially for the recurrence
arguments, we make all of these trees as balanced as possible, so that a
vertex of degree $d$ (star) is replaced by a tree of diameter $O(\log d)$.

To make this precise, we first cut every edge in half, so that
every vertex becomes a star with $d$ leaves.  Next, each such star is
replaced by a balanced tree with $d-2$ internal vertices of degree $3$ and
$d$ leaves.  The leaves are in bijection with the leaves of the star that the
tree is replacing, in cyclic order.  The leaves are identified as
in the original map with leaves on other trees.  This creates a map with
maximal degree $3$.  (The new map is not $3$-regular, since vertices
of degree $1$ or $2$ maintain their degree and identified leaves have
degree 2.)  The choices of tree for each
vertex is arbitrary, except for being maximally balanced.  See
\cref{fig:startree} for an illustration. 

\begin{figure}
  \centering
  \includegraphics[width = 0.8\textwidth]{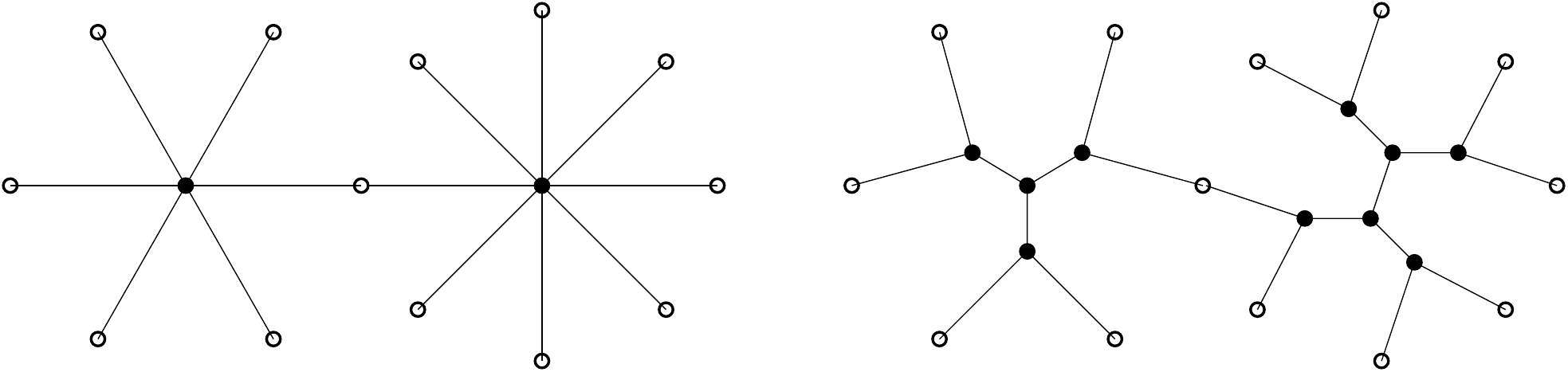}
  \caption{The star tree transform.  The white circles are the vertices
    created when cutting edges in half.  Here, vertices of degree 6 and 8
    are replaced by balanced binary trees with the same number of leaves.}
  \label{fig:startree}
\end{figure}

When the star-tree transform is applied to a map $G$, we call the resulting
map $G'$.  Clearly $G$ is a minor of $G'$, as it can be recovered by
contracting each tree back to a single vertex.  A vertex of degree $d$ in
$G$ corresponds to $(d-2)\vee 1$ vertices in $G'$.  Edges in the map $G'$
are now assigned conductances.  All edges of a tree associated with a
vertex of degree $d$ are given conductance $w_e=d$.  This allows us to use the
following lemma.

\begin{lemma}[\cite{GN12}]\label{L:equivalent}
  Let $G$ be a planar map, and $G'$ the weighted star-tree transform of $G$.
  If $G'$ is recurrent, then so is $G$.
\end{lemma}

For a rooted map, we can give the transformed map $G'$ a root $\rho'$ by
choosing uniformly a root within the tree (including the leaves)
corresponding to $\rho$.

Recall the rooted graph $(M_n,\rho_n)$ from \cref{sec:convergence}, where
some vertices are the skeleton. Apply the star-tree transform to $M_n$ to
get $M'_n$.
The skeleton vertices of $M_n'$ are all vertices in trees associated with
skeleton vertices of $M_n$, including the leaves.

There are two ways to choose a root for $M'_n$.  First, we could choose a
root uniformly among all skeleton vertices of $M'_n$.  The law of the
resulting rooted map is denoted $\nu_n$.  Take an arbitrary subsequential
limit of $\nu_n$ and call it $\nu$.  A second way is to take the rooted
$(M_n,\rho_n)$, and take the root of $M'_n$ to be a uniform vertex from the
tree associated with $\rho_n$.  We call the law of this rooted map $\mu_n$.
Note that the star tree transform is continuous in the local topology.
Since $(M_n,\rho_n)$ converges to $(M,\rho)$ we have that $(M'_n,\rho'_n)$
converges to $(M',\rho')$, with law $\mu=\lim \mu_n$, where $\rho'$ is a
uniform vertex in the tree associated to the root $\rho$ of $M$.

\begin{lemma}\label{lem:ac}
  The measure $\mu$ is absolutely continuous with respect to $\nu$.
\end{lemma}

\begin{proof}
  Given $M_n$, each skeleton vertex is equally likely to be the root under
  $\nu_n$.  Under $\mu_n$ a skeleton vertex in the tree of a vertex $v\in
  M_n$ has probability proportional to $(2\deg(v)-2)^{-1}$ of being the
  root, since we need to choose $\rho_n=v$ and the associated tree has
  $2\deg(v)-2$ vertices.  Thus the Radon-Nikodym derivative $d\mu_n/d\nu_n$
  is proportional to $(2\deg(\rho_n)-2)^{-1}$.  Since every skeleton
  vertex has degree at least $2$, $d\mu_n/d\nu_n \le 1/2$. In the case of an
  identified leaf between skeleton vertices $u$ and $v$, the probability is
  proportional to $(2\deg(u)-2)^{-1} + (2\deg(v)-2)^{-1} \le 1$. Thus
  using dominated convergence, 
  $\mu$ is absolutely continuous with respect to $\nu$. 
\end{proof}

Finally in order to use circle packing, it is useful to work with
triangulations. We triangulate each face of
$M_n',M'$ to obtain a triangulation.
This can be done while maintaining bounded degrees, as in \cite{BeSc}.
By a slight abuse of notation, we also denote the resulting maps
by $(M_n',\rho_n')$ and $(M',\rho')$ and their law by $\nu$.
Since adding edges can only makes a graph transient (via
the Rayleigh Monotonicity Principle), we immediately deduce the following
using \cref{L:equivalent,lem:ac}.

\begin{corollary}
  If $(M',\rho')$ is $\nu$-almost surely recurrent, then $(M,\rho)$ is
  almost surely recurrent, as is its subgraph $H$.
\end{corollary}

Finally, we shall also need a simple lemma relating adjacency in $M_n$ and
$M'_n$.  Let $\pi:M'_n\to M_n$ be the projection mapping each vertex in the
tree corresponding to a vertex $v$ to $v$.  A vertex arising from the
splitting of an edge in two is mapped (arbitrarily) to one of the two
endpoints of the edge.

\begin{lemma} \label{L:topology}
  If $u\sim v$ in $M'_n$, then either $\pi(u)=\pi(v)$ or else
  $\pi(u)\sim\pi(v)$.
\end{lemma}

\begin{proof}
  Since $M_n$ is a triangulation, after vertices are replaced by trees,
  each face consists of paths from three trees corresponding to a face of
  $M_n$, and three vertices corresponding to the edges between the trees.
  All additional edges in $M'_n$ connect vertices within a face.
\end{proof}

%%%%%%%%%%%%%%%%%%%%%%%%%%%%%%%%%%%%%%%%%%%%%%%%%%%%%%%%%%%%%%%%%%
\section{Recurrence via circle packing}
\label{sec:recurrence}

All the tools are in place, and we are ready to
build on the methods of \cite{BeSc,GN12} to prove our main
result. Throughout this section we have maps $(M_n',\rho_n')$ and $(M',\rho')$
with law $\nu_n$ and $\nu$ respectively.

Let us recall some useful terminology.  Given a set of points
$\cC$ in a metric space, the radius of isolation $R_x$ of a point $x\in\cC$
is the minimal distance to another point of $\cC$.  Following \cite{BeSc},
we say that a point 
$x\in\cC$ is \textbf{$(\delta,s)$-unsupported} if all but $s$ of the points
in $B(x,\delta^{-1}R_x)$ can be covered by a ball of radius $\delta R_x$.
Otherwise it is \textbf{$(\delta,s)$-supported}.
  A key idea in \cite{BeSc}, is that for small $\delta$ and large
$s$, a finite set cannot have too many $(\delta,s)$-supported points.
We use a quantitative form of this appears in \cite[Lemma~3.4]{GN12}:

\begin{lemma}[\cite{GN12}]\label{lem:delta_s}
  There exists some $A$, so that for any finite $\cC \subset \R^2$,
  for all $\delta \in (0,1/2)$ and $s \ge 2$, the fraction of
  $(\delta,s)$-supported points in $\cC$ is at most $\frac{A\log
    (\delta^{-1})}{\delta^2 s}$.
\end{lemma}

In previous work, this lemma was applied to the set of centres of a circle
packing of a given graph.  A key difference from previous work, is that we
take the set $\cC$ to be the set of centres of the circles corresponding to
skeleton vertices, and not all vertices.  Let $P_n$ be some (arbitrarily
chosen) circle packing of $M'_n$ in $\R^2$ (which exists in light of the
Circle Packing Theorem \cite{K36}).  Since $M_n'$ is a bounded degree
triangulation (with boundary), we may take $P_n$ so that ratios of radii of
adjacent circles are bounded.

Having fixed some circle packing for $M'_n$,  we now consider the uniform
skeleton root $\rho'_n$.  Apply a translation and dilation to $P_n$ so
that the circle corresponding to the root $\rho'_n$ is the unit disc, and
let $Q$ be the image of $\cC$ after this transformation, which is now
defined on the same probability space as $H$, $M_n$ and $M'_n$.

\begin{lemma}\label{lem:skel_cp}
  Let $E_r$ be the event that all but $r^3$ points of $Q\cap\{|z|<r\}$ can be
  covered by a disc of radius $r^{-1}$.  There exists some $A$, such that
  for all $r \ge 2, n \ge 1$ we have $\P(E_r) > 1-\frac{A\log r}{r}$.
\end{lemma}

\begin{proof}
  an arbitrary $M'_n$, and take $\cC$ to be the set of the centres of
  circles of skeleton points in $M'_n$.  For a uniform vertex $v$, scale so
  that the circle of $v$ is the unit circle.  By the Ring Lemma, the radius
  of isolation $R_v$ is in $[1,C]$ for some absolute constant $C$.
  
  Now apply \cref{lem:delta_s} with $s=r^3$ and $\delta = 1/(Cr)$.  We find
  that if $v$ is uniform in $\cC$, with the claimed high probability, all
  but $r^3$ points in $\cC\cap \{|z|\leq Cr R_v\}$ can be covered by a disc
  of radius $R_v/Cr$. Since $Cr R_v \geq r$ and $R_v/Cr \leq 1/r$, this
  implies the claim.
\end{proof}

Consider the subgraph of $M'_n$ induces by vertices in $\{|z|<r\}$.  Let
$\Gamma(n,r)$ denote the connected cluster of $\rho'_n$ in this graph, and
let $\bar\Gamma(n,r)$ denote $\Gamma(n,r)$ together with all edges
connecting the cluster to vertices outside $\{|z|>r\}$.  A major step in
our proof is to show that for some constant $\alpha$, the resistance in $\bar\Gamma(n,r)$ from
$\rho'_n$ to the complement of $\{|z|>r\}$ is at least 
$\alpha$ with high probability.  Of course, this is the same as the
resistance in $M'_n$ between the same vertex sets.  Moreover, we shall
prove all this not just for the resistance from $\rho'_n$, but from any
finite neighbourhood of $\rho'_n$, i.e.\ there is some $\alpha>0$ so that
for any finite set the resistance from the set to the complement of
$\{|z|>r\}$ is at least $\alpha$ for $r$ large enough.
Towards this, we first prove that (with
high probability) the maximal conductance of any edge in $\bar\Gamma(n,r)$
is at most $C\log r$, and that if all conductances are changed to $1$ then
the resistance between the involved vertex sets are at least $c\log
r$. (Recall that the conductance of an edge is the degree of the vertex
corresponding to it before the star-tree transform.)
The claim then follows by Rayleigh monotonicity with $\alpha=c/C$.  In what
follows, $\Res(A,B; w)$ denotes the resistance from $A$ to $B$ with edge
weights $w$.  The graph is implicit and should be clear from the context.

\begin{lemma}\label{L:high_resistance}
  Fix $k$, and let $B_k\subset M'_n$ be the ball of graph radius $k$ around
  $\rho'_n$.  For some $c$, for all $r$ large enough, in $\bar\Gamma(n,r)$
  we have 
  \[
  \Res(B_k,\{|z|\geq r\}; 1) \geq c_1\log r.
  \]
\end{lemma}

\begin{proof}
  The radius of the circle of $\rho'_n$ is $1$.  By the Ring Lemma, radii
  of adjacent circles have bounded ratio, so every vertex of $B_k$ is
  contained in $\{|z|<r'\}$ for some $r'=r'(k)$.  The resistance across the
  annulus $\{r'<|z|<r\}$ is now seen to be at least $c\log(r/r')$ by the
  arguments of \cite{HS93,BeSc} (see for example \cite[Corollary~3.3]{GN12}).
\end{proof}

\begin{lemma}\label{L:max_conductance}
  Let $w_{\max}$ denote the maximal conductance of any edge in
  $\bar\Gamma(n,r)$.  Then for some $c_0$ we have
  \[
  \lim_{r\to\infty} \limsup_{n\to\infty} \P(w_{\max} \geq c_0 \log r) = 0.
  \]
\end{lemma}

\begin{proof}
  Fix $\eps>0$, and consider the event $E_r$ of \cref{lem:skel_cp}.  For
  $r\geq r_0(\eps)$ we have $\P(E_r) \geq 1-\eps$.  Assume $r\geq 1$ and
  that $E_{2r}$ holds, and let $U=\{|z-z_0|<1/2r\}$ be a disc such that
  $\{|z|<2r\}\setminus U$ contains at most $(2r)^3$ skeleton vertices.

  We consider several possibilities according to the location of $U$.  If
  $U$ is disjoint of $\{|z|<r\}$, then $\{|z|<r\}$ contains at most $8r^3$
  skeleton vertices.  Otherwise, $|z_0|<r+1/2$ (since $r\geq 1$).  Suppose
  $U$ contains at least $2$ vertices, which therefore have circles of
  radius at most $1/r$.  Let $U_a = \{|z-z_0| < a/r\}$.  From the Ring
  Lemma it follows that for some $a$, the vertices in the annulus
  $U_a\setminus U$ disconnect $U$ from the complement of $U_a$.  In that
  case, since $M'_n$ is a triangulation, there is a cycle in that annulus
  that surrounds $U$.  For $r$ large enough, this cycle lies in $\{|z|<2r\}
  \setminus U$, and so it contains at most $8r^3$ skeleton vertices (and
  possibly more non-skeleton vertices).

  Let us summarize our findings so far.  For $r$ large enough and any $n$,
  with probability at least $1-\eps$, there is a set of at most $8r^3$
  skeleton vertices in $M'_n$ that contains every skeleton vertex in
  $\{|z|<r\}$ except possibly those in $U$.  If any vertices from $U$ are
  missed, the set also contains all skeleton vertices from a cycle
  separating $U$ from $\rho'_n$.  That cycle need not be contained in
  $\{|z|<r\}$. See \cref{F:cpproof}.

  \begin{figure}
    \centering
    \includegraphics[width = 0.3\textwidth]{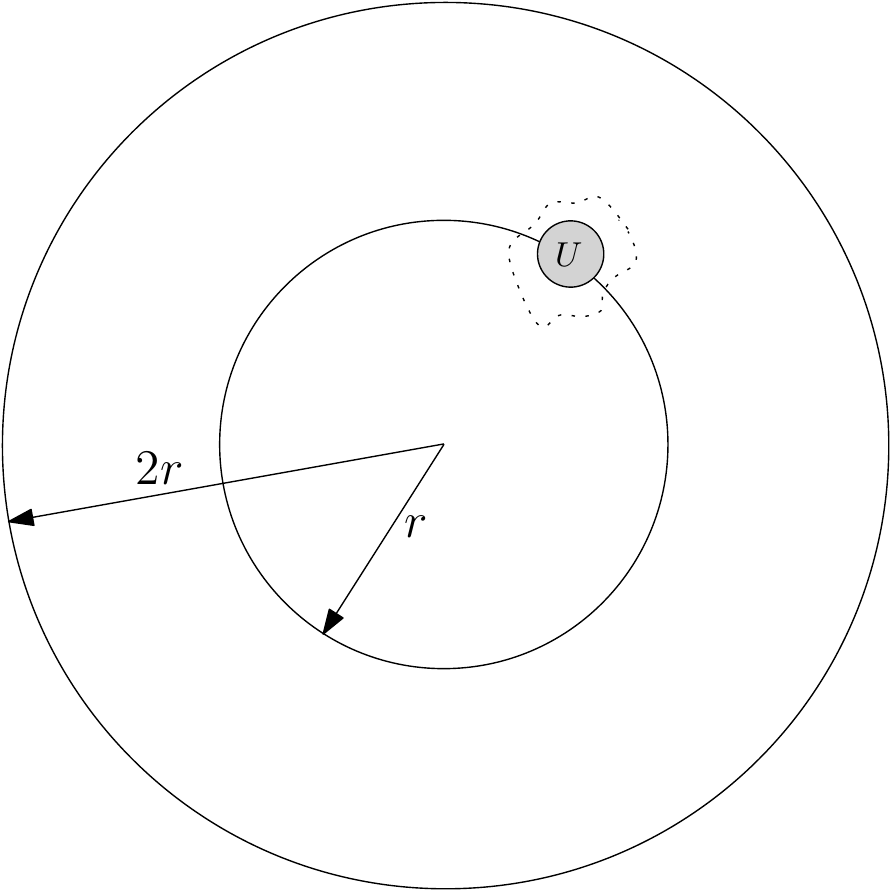}
    \caption{Illustration of the proof of \cref{L:max_conductance}.  The
      grey disc $U$ may include an arbitrary number of skeleton vertices,
      but the rest of the large ball, including the cycle around $U$
      contain at most $8r^3$ skeleton vertices.  The cycle is needed if $U$
      intersects but is not surrounded by the ball of radius $r$.}
    \label{F:cpproof}
  \end{figure}
  
  Now, any path $\gamma$ in $M'_n$ which does not contain any boundary
  vertex of $M_n'$ projects via $\pi$ to a path in $M_n$.  The restriction
  of $\pi(\gamma)$ to the skeleton vertices is a path in $\Sk(M_n)$, which
  visits no more skeleton vertices than $\gamma$ (by \cref{L:topology}).
  For any $r$, for large enough $n$ with probability $1-\eps$, no boundary
  vertex of $M_n'$ is in $\bar \Gamma(n,r)$ since otherwise, the skeleton
  distance from $\rho'_n$ to a boundary vertex is at most $8r^3$
  (\cref{lem:root_inside}).  Thus for any $r$, for large enough $n$, with
  probability $1-2\eps$, $\Gamma(n,r)$ is contained in the hull of
  $B_\Sk(\rho'_n,8r^3)$.  The result now follows from \cref{lem:DmaxM}.
\end{proof}

As noted, by combining \cref{L:high_resistance,L:max_conductance} we get
the following with $\alpha=c_1/c_0$.

\begin{prop}\label{P:const_res}
  Fix an integer $k$ and $\eps>0$, and let $B_k\subset M'_n$ be the
  ball of graph distance $k$ around
  $\rho'_n$.  For some $\alpha>0$, for all $r$ large enough, we have with
  probability at least $1-\eps $ as $n\to\infty$
  \[
  \Res(B_k,\{|z|\geq r\}; w) \geq \alpha.
  \]
\end{prop}

\begin{proof}[Proof of \cref{thm:main,thm:M_rec}]
  The argument is similar to the argument of \cite{GN12}.  We start with
  the observation that an electrical network $G$ is recurrent if and only
  if for some $\alpha>0$, for every graph distance ball $B_k=B_k(\rho)$
  there exists a finite vertex set $S$ such that
  \[
  \Res(B_k, G \setminus S; w) > \alpha.
  \] 
  Fix $k$, and $\eps>0$.  By \cref{P:const_res}, for any large enough $n$,
  with probability $1-\eps$ there is some finite $S$ such that in $M'_n$ we
  have $\Res(B_k, M'_n \setminus S; w) > \alpha$.  Moreover, with high
  probability for some $R$, the set $S$ is contained in the hull of a ball
  of radius $R$ 
  in $M'_n$.  Going to the limit, we find that for $n$ large enough, with
  probability at least $1-2\eps$ the resistance in $M'$ from $B_k$ to the
  complement of some large finite set $S$ is at least $\alpha$.  Since $\eps$ is
  arbitrary, this implies that $M'$ is $\nu$-almost surely recurrent.

  By \cref{lem:ac}, this implies that $M'$ is $\mu$-almost surely
  recurrent, which in turn also implies recurrence of $M$, and of $H$.
\end{proof}

\section{Extensions}\label{sec:extens-open-quest}

\paragraph{Resistance estimates.}
From the argument above we also get some explicit estimates on the growth
of the resistance in $M$.  In the annulus between Euclidean radii $2^n$ and
$2^{n+1}$ the maximal degree is of order $n$. Since the resistance across
the annulus without weight is at least $C$, this indicates that the
resistance to distance $2^n$ is at least $c\log n$, i.e.\ the resistance to
Euclidean distance $R$ is $\log\log R$.  This argument can be made precise,
but we do not pursue this here.  It would be interesting to get better
bounds on the growth of the resistance (it is believed to grow like
$\log$).

\paragraph{Other classes of maps.}
One natural generalization is to consider uniform infinite
domain Markov half plane triangulations with self-loops. Such
triangulations can be obtained by taking a HUIPT and decomposing every edge into i.i.d.\
Geometric number of edges and attaching self-loops on one of the
vertices in the $2$-gons thus formed by tossing a fair coin  (see
\cite{AR13} for detailed discussion on this.) Note that a self-loops with
any finite triangulation inside it do not effect recurrence
or transience so we can delete them. We can now form an equivalent
network by collapsing the geometric number of multiple edges into a
single edge and giving this edge a conductance which is equal to the
number of edges combined to form it.  Thus the equivalent network is
HUIPT but with i.i.d.\ geometric conductances on each edge.  It can be
checked by the diligent reader that our analysis of the HUIPT goes
through in this case also, implying recurrence of this case as well.

A more difficult problem is proving recurrence of more general half planar
maps.  It is easy to see that a layer decomposition is still possible for
various other classes of half plane uniform infinite maps. For
quadrangulations, a similar layer decomposition was introduced by Krikun in
\cite{krikun2005local}.  The main estimate needed is that the maximal
degree in the skeleton balls grow logarithmically in the radius (an
analogue of \cref{lem:DmaxM}.) For maps with even larger faces, a layer
decomposition is still possible but it becomes more complicated.

\paragraph{Hyperbolic maps.}
A one parameter family of hyperbolic versions of the half plane
UIPT were constructed in \cite{AR13}. A full plane hyperbolic version
was constructed in \cite{PSHIT} and it was shown in \cite{ANR14} that
the half plane versions can be be realized as a sub-map of the full plane
ones. One can carry out the layer decomposition and a full plane
extension of the half plane maps in
exactly the same way as done in this paper. Call such a full plane map
$M^{\text{hyp}}$. The volume of the triangulation inside the holes in this
situation will have exponential tail. It is not too difficult to
see that the lower half of the triangulation in  $M^{\text{hyp}} $ is
recurrent.  In this situation, if we look at the sequence of
hulls of radius $r$ and uniformly pick a vertex, the map converges
locally to some rerooted version of the lower half, and so the maps are a
local limit of finite planar graphs with exponential degree distribution.
Exploring the connection between $M^{\text{hyp}}$ and the full plane map
defined in \cite{PSHIT} is also of interest.

\paragraph{Stationarity.}
It is easy to see that if we put appropriate conductances on the edges of
$\Sk(M)$ and bias by the degree (in $ M$) of the root vertex, we obtain a
stationary reversible graph. A similar construction can be carried out for
the hyperbolic versions to obtain $\Sk(M^{\text{hyp}})$. For a simple
random walk $Y_0,Y_1,\ldots$ in $\Sk(M^{\text{hyp}})$ or $\Sk(M)$, if we
let $\ell(Y_i)$ denote the index of the layer below $Y_i$ an application of
ergodic theorem lets us conclude
\begin{equation}
  \label{eq:2}
  \frac{\ell(Y_i)}{i} \to s
\end{equation}
almost surely for some constant $s$. It follows from the results in
\cite{ANR14} that $s>0$ almost surely in $\Sk(M^{\text{hyp}})$ and the
recurrence result in this paper shows $s=0$ almost surely for
$\Sk(M)$. Notice that simple random walk in $M^{\text{hyp}}$ spends a
positive fraction of its time in the skeleton vertices (this is easy
to see again via stationarity and exponential tail of the volume of
the holes). From all this we can deduce the existence of the speed of
simple random walk away from the boundary in $H$. This answers a
question in \cite{ANR14} where only positive liminf speed from the
boundary was established.

\subsection*{Acknowledgments}

This work begun while GR was at UBC, and completed during a visit of OA to
the Isaac Newton Institute.  OA was partly supported by NSERC, the Isaac
Newton Institute and the Simons Foundation. GR was supported by the Engineering and Physical Sciences Research Council under grant EP/103372X/1.

%%%%%%%%%%%%%%%%%%%%%%%%%%%%%%%%%%%%%%%%%%%%%%%%%%%%%%%%%%%%%%%%%%

\bibliographystyle{abbrv}
\bibliography{HUIPTrec_arxiv}

\end{document}